\documentclass[a4paper]{amsart}
\usepackage{amsmath,amsthm, amssymb}
\usepackage{graphicx}
\usepackage{graphics}

\usepackage{mathtools}

\newcommand{\R}{{\mathbb R}}
\newcommand{\ds}{\displaystyle}
\newcommand{\es}{\mathop{\rm essinf}\limits}

\newcommand{\ol}{\overline}

\newcommand{\B}{{\mathcal  B}}

\newcommand{\M}{{\mathcal M}}

\newcommand{\ba}{{\mathbf a}}
\newcommand{\bb}{{\mathbf b}}
\newcommand{\bc}{{\mathbf c}}
\newcommand{\bF}{{\mathbf F}}

\newcommand{\divv}{{\rm div\,}}
\newcommand{\tl}{\tilde}
\newcommand{\diam}{{\rm diam\,}}

\newcommand{\cP}{{\mathcal P}}
\newcommand{\bG}{{\mathbf  G}}

\numberwithin{equation}{section}

\def\Xint#1{\mathchoice
    {\XXint\displaystyle\textstyle{#1}}%
     {\XXint\textstyle\scriptstyle{#1}}%
     {\XXint\scriptstyle\scriptscriptstyle{#1}}%
     {\XXint\scriptstyle\scriptscriptstyle{#1}}%
    \!\int}
\def\XXint#1#2#3{{\setbox0=\hbox{$#1{#2#3}{\int}$}
    \vcenter{\hbox{$#2#3$}}\kern-.5\wd0}}

\newtheorem{thm}{Theorem}[section]  
\newtheorem{lem}[thm]{Lemma}           
\newtheorem{defin}[thm]{Definition}    

\begin{document}



\title[ Asymptotically regular operators in generalized Morrey spaces]{ Asymptotically regular operators in generalized Morrey spaces}

\author[S.-S. Byun]{Sun-Sig Byun}
\author[L. Softova]{Lubomira  Softova}

\begin{abstract}
We   obtain Calder\'on-Zygmund  type estimates in generalized Morrey
spaces for nonlinear equations   of  $p$-Laplacian type.   Our
result is obtained under minimal regularity assumptions both on the
operator  and on the   domain. This result allows us to study
asymptotically regular operators. As a byproduct, we obtain also
generalized H\"older regularity of the solutions under some
minimal restrictions of the weight
functions.
\end{abstract}

\maketitle

\section{Introduction}\label{sec1}

In the  present work we obtain {\it  global  Calder\'on-Zygmund  type estimates}   in
  {\it generalized Morrey spaces} $L^{q,\varphi}(\Omega)$  for operators of the kind
$$
\divv \ba(x,Du)=\divv(|\bF|^{p-2}\bF)\qquad \text{ in } \    \Omega
$$
for some fixed $p\in(1,\infty)$ and bounded domain $\Omega $ with a
rough boundary. Supposing first that $\ba(x,\xi)$ is
 a {\it  regular elliptic operator of $p$-Laplacian type}, we show, under minimal regularity assumptions, the implication
\begin{equation}\label{eq_impl}
|\bF|^p\in L^{q,\varphi}(\Omega)\  \Longrightarrow \  |Du|^p\in L^{q,\varphi}(\Omega)
\end{equation}
for each $q\in[1,\infty)$ and   measurable function $\varphi(x,r)$
satisfying suitable conditions.  Recall that the operator
$\ba(x,\xi)$ is {\it regular}  if it is {\it differentiable and
monotone}  with respect to $\xi\in \R^n$ for a.a. $x\in \R^n.$ The
implication \eqref{eq_impl} has
been  proved first by Calder\'on and Zygmund  for the Poisson
equation and in the framework of  Lebesgue spaces. Since then, there
have been  a lot of works
treating  the validity of this
implication for linear and nonlinear operators and in various
function spaces, see for instance  \cite{B, BP, BP2, BR, BS, BW1,
BW, MP, PS} and the references therein.

In \cite{CE} Chipot and Evans introduced the notion of {\it
asymptotically regular operators} in the elliptic framework. Later,
this notion has been extended to {\it  asymptotically regular
operators with $p$-growth} (see \cite{Ry}). A local
Calder\'on-Zygmund theory and partial Lipschitz regularity for
asymptotically regular elliptic systems have been developed in
\cite{SS1, SS2}, while global results have been
obtained in \cite{F}. These studies
have been extended later to operators satisfying weaker asymptotic
regularity condition of $p$-growth
that allows to consider nonlinear
problems oscillating around some regular problems (see \cite{BOW}).

In the present work,  we  develop
the Calder\'on-Zygmund theory for such operators in the settings of
{\it generalized Morrey spaces}
$L^{q,\varphi}.$   For the description and some properties of these
spaces, see for instance \cite{BS,
Na} and the references therein. It is well known from the classical
theory that if the gradient of a
function belongs to a certain
Morrey space, then the embedding theorems of Morrey and Campanato
(cf. \cite{Cm,Morrey1} imply H\"older regularity of the
function. In \cite{CP} the authors
obtain an analogy of the classical
Sobolev-Campanato embedding theorems but in the framework of
generalized spaces. They prove sharp inclusion relations among
generalized Morrey, Campanato and H\"older spaces,  considering
continuous  weight function $\varphi(r):\R_+\to \R_+,$  satisfying
appropriate  growth conditions.  Restricting
a class of weight functions, we  show as a byproduct that the
Calder\'on-Zygmund estimate implies {\it generalized  H\"older
regularity}  of the same solution.

As mentioned above, we are going to
study the Dirichlet problem for operators ``close" to some regular
operator under minimal regularity assumptions
on the nonlinearity and the domain. For
this goal, we consider domains with boundary satisfying the
Reifenberg condition of flatness \cite{R}. This condition allows to
the boundary to be very rough such that even the unit normal vector
cannot be defined, but it should be flat enough such that it can be
well approximated by $n-1$-dimensional planes.  This structural
condition implies the validity of the internal and external cone
conditions and hence the validity of the fundamental results from
the functional analysis. An example of such a boundary can be {\it
the Koch snowflake} with an angle of the spike $\beta$ such that
$0<\sin\beta<\delta< \frac{1}{8}.$
 A detailed overview of the properties of these domains can be found  in \cite{MT,Toro}.

Throughout the paper, the letter $c$ will denote  a universal
constant that can be explicitly computed in terms of known
quantities. The exact value of $c$ may vary from one occurrence to
another. As usual, $D_i u=\frac{\partial u}{\partial x_i}$, $1 \leq
i \leq n$, and $Du=(D_1u,\ldots,D_nu)$ means the gradient vector
field of $u$. For any bounded domain $\Omega\subset\R^n$, we denote
by $|\Omega|$ the Lebesgue measure of $\Omega$ and by $d_\Omega$ the
diameter of $\Omega.$  Let $\B_r(y)$ be the ball centered in $y\in
\Omega$ and of radius $r\in(0,d_\Omega],$ then
$\Omega_r(y)=\Omega\cap \B_r(y).$ In addition the repeated-index
summation convention is adopted.

\section{Regular $p$-Laplace type  problems}\label{secRP}

Let $\Omega\subset \R^n$,  $n\geq 2$,  be a bounded domain and $p\in(1,\infty)$ be fixed.  We
consider the following Dirichlet  problem
\begin{equation}\label{DP}
\begin{cases}
\divv \ba(x,D u) = \divv (|\bF|^{p-2}\bF) &  \text{ in } \Omega,\\
  u(x)=0 & \text{ on } \partial \Omega,
\end{cases}
\end{equation}
where $\bF=(f^1,\ldots,f^n)\in L^p(\Omega;\R^n)$ is a given
vector-valued function and $\ba =\ba(x,\xi):\R^n\times \R^n\to \R^n$
is a Carath\'eodory map, i.e.,
measurable in $x$ for each $\xi$ and continuous  in $\xi$ for a.a.
$x\in \Omega$.

\begin{defin}\label{defRegular}
The operator $\ba(x,\xi)$  is   regular  if  it    is     differentiable in $\xi$  and   satisfies the following  structural conditions
\begin{equation}\label{str-cond}
\begin{cases}
\ds\gamma |\xi|^{p-2}|\eta|^2\leq \langle D_\xi \ba(x,\xi)\eta,\eta   \rangle,\\[5pt]
\ds |\ba(x,\xi)| + |\xi||D_\xi\ba(x,\xi)|\leq \Lambda |\xi|^{p-1}
\end{cases}
\end{equation}
for each  $0 \neq \xi$, $\eta\in \R^n$, for a.a.   $x\in \R^n,$ and
for some positive  constants $\gamma$ and $\Lambda.$
\end{defin}
The condition \eqref{str-cond} easily implies {\it monotonicity } of
$\ba,$ i.e.,
\begin{align}\label{cond3}
\nonumber
\langle  & \ba(x,\xi)-\ba(x,\eta),  \xi-\eta  \rangle\\[4pt]
&\   \geq
\gamma_1\begin{cases}
\ds|\xi-\eta|^p & \text{  if }  p\geq 2\\[4pt]
\ds |\xi-\eta|^2(1+|\xi|+|\eta|)^{p-2} &   \text{  if   } 1<p<2,
\end{cases}
\end{align}
where $\gamma_1$ depends only on $\gamma, n$ and $p$.

Recall that  a {\it weak solution} of \eqref{DP}  means  a
function $u\in W_0^{1,p}(\Omega),$  $p\in(1,\infty)$ that satisfies
$$
\int_\Omega    \langle \ba(x, Du), D\phi\rangle\, dx= \int_\Omega  \langle |\bF|^{p-2}\bF, D\phi\rangle\, dx
$$
for any $\phi\in W_0^{1,p}(\Omega).$

The {\it unique weak solvability } of \eqref{DP} follows by the  {\it Minty-Browder method } in  $L^p$ that gives the   estimate
\begin{equation}\label{eq1}
\|Du\|_{L^p(\Omega)}\leq c \|\bF\|_{L^p(\Omega)},
\end{equation}
where the constant $c$ depends only on $n$, $p$, $\gamma$, $\lambda$
and $|\Omega|$, see for instance \cite{Ba, Di}.

We suppose that the dependence of $x$ in the nonlinear term $\ba$ is
of small $BMO$ ({\it bounded mean oscillation}) type. In order to describe this regularity, we have to
be able to measure the oscillation of the mapping $x \rightarrow
\frac{\ba(x,\xi)}{|\xi|^{p-1}}$  over balls, uniformly in $\xi \in
\R^n \setminus \{0\}$. To this end, we introduce the function
\begin{equation}\label{theta}
\theta(\ba;r,y)(x)=\sup_{\xi\in \R^n\setminus \{0\}}
\frac{|\ba(x,\xi)-\ol \ba_ {\B_r(y)} (\xi) |} {|\xi|^{p-1}},
\end{equation}
where
$
\ol\ba_{\B_r(y)}(\xi)=\frac1{|\B_r(y)|} \int_{\B_r(y)}\ba(x,\xi)\,
dx
$
is the integral average of $\ba$ over $\B_r(y)$  with respect to  $x$
for any fixed $\xi\in \R^n \setminus\{0\}$.

\begin{defin}
We say that the vector field $\ba$ is $(\delta,R)$-vanishingif
\begin{equation}\label{vanishing}
\sup_{0<r\leq R}\sup_{y\in \R^n}\Xint-_{\B_r(y)} \theta(\ba;
r,y)(x)\, dx\leq \delta.
\end{equation}
\end{defin}

As it concerns the domain $\Omega$, we  impose the following kind of flatness.
\begin{defin}
We say that $\Omega$ is $(\delta,R)$-Reifenberg flat if for every $x_0\in \partial \Omega$ and every $r\in(0,R],$ there exists a coordinate system $\{y_1,\ldots, y_n\}$ centered in $x_0,$ which can depend on $r$ and $x_0,$ so that $x_0=0$ with respect to it and
\begin{equation}\label{Reifenberg}
\B_r(0)\cap\{y_n>\delta r\}\subset B_r(0)\cap \Omega\subset B_r(0)\cap \{y_n>-\delta r\}\,.
\end{equation}
\end{defin}

The problem \eqref{DP} is invariant under scaling and normalization,
as it can be seen by the following lemma (see \cite[Lemma~2.5]{BR}).

\begin{lem}\label{lem1}
 For each $\lambda \geq 1$ and $0<\rho \leq 1$,  we
define the rescaled maps:
\begin{align*}
&\tl \ba(x,\xi)=\frac{\ba(\rho x, \lambda\xi)}{\lambda^{p-1}},\quad \tl\Omega=
\left\{\frac{x}{\rho}: x\in \Omega   \right\},\\
&\tl u(x)=\frac{u(\rho x)}{\lambda x}, \quad \tl \bF(x)=\frac{\bF(\rho
x)}{\lambda}.
\end{align*}
Then
\begin{enumerate}
\item  $\tl u\in  W_0^{1,p}(\tl\Omega)$ is the weak solution of
$$
\divv \tl \ba (x,D\tl u)= \divv( |\tl \bF|^{p-2} \tl \bF)\qquad \text{ in }  \   \tl\Omega.
$$
\item $\tl\ba$ satisfies the structural conditions \eqref{str-cond} with the same constants $\Lambda, \gamma.$
\item $\tl\ba$ is $(\delta, \frac{R}{\rho})$-vanishing.
\item $\tl\Omega$ is $(\delta,\frac{R}{\rho})$-Reifenberg flat.
\end{enumerate}
\end{lem}

Note that thanks to the scaling invariance property, one can take
for simplicity $R=1$ or any other constant bigger than or equal to
$1$. On the other hand, $\delta$ is a small positive constant, say
$0<\delta<1/8$, being invariant under such a scaling argument.

We suppose that $|\bF|^p\in L^{q,\varphi}(\Omega)$ with
$q\in(1,\infty)$  and $\varphi$ satisfying
\eqref{cond1}-\eqref{cond5}, as
described in Section 3 below, which implies that $ |\bF|^p\in
L^{1}(\Omega).$  In fact,
\begin{align*}
\||\bF|^p\|^q_{L^q(\Omega)}=\   &
\varphi(y,r^\ast) \frac1{\varphi(y,r^*)}   \int_\Omega |\bF(x)|^{pq}\,dx \\
\leq\  &   \varphi(0,d)
\||\bF|^p\|^q_{L^{q,\varphi}(\Omega)}.
\end{align*}
Then the H\"{o}lder inequality implies
\begin{equation}\label{existence}
\||\bF|^p\|_{L^1(\Omega)}  \leq |\Omega|^{1-\frac1q}\|
|\bF|^p\|_{L^q(\Omega)} \leq c \||\bF|^p\|_{L^{q,\varphi}(\Omega)},
\end{equation}
which ensures the existence of a
unique weak solution $u\in W_0^{1,p}(\Omega)$  of  \eqref{DP}, where
the constant $c$ depends on $n,p,\varphi, |\Omega|$. Moreover, it is
shown in \cite{BR} that this solution belongs to $W_0^{1,q}(\Omega)$
and we have the following estimate
\begin{equation}\label{estBR}
\||Du|^p\|_{L^q(\Omega)}\leq
c\||\bF|^p\|_{L^q(\Omega)}.
\end{equation}

Our goal is to develop the Calder\'{o}n-Zygmund  theory  for the
problem (\ref{DP}) in the setting  of the generalized Morrey spaces.
Namely,  taking $|\bF|^p\in L^{q,\varphi}(\Omega)$ with
$q\in(1,\infty)$  and $\varphi$ satisfying suitable doubling and
integral conditions, we are going to show that the gradient $|Du|^p$
belongs to the \textit{same} space $L^{q,\varphi}(\Omega)$   with the desired estimate
\eqref{mainestimate} below, which is a correct and natural extension
of \eqref{estBR} in Lebesque spaces to the one in generalized Morrey
spaces.

\section{Generalized Morrey type regularity}\label{sec2}

Let $\varphi=\varphi(y,r):
\R^n\times \R_+ \to \R_+$ be a measurable function
and $1<q<\infty$. The {\it
generalized   Morrey space} $L^{q,\varphi}(\Omega)$ consists of all
measurable functions $f $ for which the following norm  is finite:
$$
\|f\|_{L^{q,\varphi}(\Omega)} =\sup_{y\in\Omega\atop 0<r\leq d_\Omega}\left(\frac1{\varphi(y,r)}
\int_{\Omega_r(y)}|f(x)|^q\,dx  \right)^{\frac1q}
$$
where   $\varphi(y,r)=\varphi(\B_r(y))$ depends on the ball.

We assume that for any fixed $y,z \in \R^n$ and
$r,s>0$, 
there are positive constants $\varkappa_1,\varkappa_2$ and $\varkappa_3,$ independent of
$y,z,r$ and  $s,$ such that
\begin{align}
\label{cond1}
  &\varkappa_1\leq \frac{\varphi(y,s)}{\varphi(y,r)}\leq \varkappa_2, \qquad r\leq s\leq 2r, \   \varkappa_1, \varkappa_2>0,\\
\label{cond3}
&\int_r^\infty \frac{\varphi(y,t)}{t^{n+1}}\, dt \leq \varkappa_3\, \frac{ \varphi(y,r)}{r^{n}}, \qquad \varkappa_3\geq 1, \\
\label{cond5}
&\varphi (y,r)\leq  \varphi(z,s)  \quad \text{ for } \quad \B_r(y)\subset \B_s(z)\,.
\end{align}

Since $\Omega $ is bounded, for any $y\in \Omega$, it holds
$
\sup_{z\in \Omega} |y-z|\leq d_\Omega\,.
$
Hence there exists $r^\ast\leq d_\Omega$ such that $\Omega\subset
\B_{r^\ast}(y) \subset \B_d(0)$ for some $d\geq d_\Omega.$   Then \eqref{cond5} gives that for all  $y\in\Omega,$
$$
\varphi(y,r^*)\leq \varphi(0,d)\,.
$$
Moreover, the monotonicity condition \eqref{cond5} implies (cf. \cite{BS})
\begin{equation}\label{cond6}
\sup_{y\in \Omega \atop{ r>0}}\frac{|\Omega_r(y)|}{\varphi(y,r)}<
\varkappa_4,
\end{equation}
which is equivalent to
$\|\chi_\Omega\|_{L^{q,\varphi}(\Omega)}<\varkappa_4,$ where
$\chi_\Omega$ is the characteristic function of $\Omega.$

\begin{thm}
\label{MainTh} Let $q\in[1,\infty)$ and $\varphi: \R^n\times \R_+
\to \R_+$ be a weight satisfying \eqref{cond1}-\eqref{cond5}. Assume
that  $\ba(x,\xi)$ is regular and $|\bF|^p\in
L^{q,\varphi}(\Omega).$
Then there exists a small positive
constant
$\delta_0=\delta_0(n,p,q,\gamma,L,\varphi)$
such that if $\ba$ and $\Omega$ satisfy \eqref{vanishing} and
\eqref{Reifenberg}, respectively with $\delta_0$,  then  $|Du|^p\in
L^{q,\varphi}(\Omega)$ and we have the following estimate
\begin{equation} \label{mainestimate}
\||Du|^p\|_{L^{q,\varphi}(\Omega)}\leq c
\||\bF|^p\|_{L^{q,\varphi}(\Omega)}
\end{equation}
with  constant $c$ depending on  known quantities.
\end{thm}
\begin{proof}
In  \cite{BR}, the authors study the weak solvability  of \eqref{DP}
in the weighted Lebesgue spaces
$L^q_w(\Omega)$ with $w$ a
{\it Muckenhoupt weight.}   Recall that $w$ belongs to the Muckenhoupt class
$A_q,$ $q\in(1,\infty)$, if
\begin{equation}\label{eqMuch}
[w]_q:=\sup_{\B}\left(\frac1{|\B|}\int_{\B}
w(x)\,dx\right)\left(\frac1{|\B|}\int_{\B}w(x)^{-\frac1{q-1}}\,dx
\right)^{q-1} <+\infty,
\end{equation}
where the supremum is taken over all balls $\B\subset \R^n$. If 
$q=1$   we say that $w\in A_1$ if
\begin{equation}
\label{eqMuch1} \frac1{|\B|}\int_{\B}w(x)\ dx\leq
[A]_1 \es_{\B} w(x)
\end{equation}
for some positive constant $[A]_1$.
By  \cite{BR}, we have that if $|\bF|^p\in L^q_w(\Omega)$ with 
 $w\in A_q,$ then $|Du|^p\in L^q_w(\Omega)$ and
verifies  the estimate
\begin{equation}\label{2.17BR}
\int_\Omega |Du|^{pq} w(x)\, dx\leq c \int_\Omega |\bF(x)|^{pq} w(x)\,dx\,.
\end{equation}
For any  Borel set $B$, the  Coifman-Rochberg result (cf. \cite{To})
asserts that    the {\it maximal operator} of the {\it
characteristic function of } $B$
$$
\M\chi_{B}(x)= \sup_{x\in \B}\frac1{|\B|}\int_{\B} \chi_{B}(z)\,dz
$$ 
verifies that $(\M\chi_B)^\sigma\in A_1$ for  any $\sigma\in (0,1).$
Because of the {\it increasing property} of the $A_q$ classes,
i.e., {\it if $w\in A_q$, then
$w\in A_p$ whenever $p>q$}, we have that $(\M\chi_{B})^\sigma\in
A_q$ for each $q>1$.

Let us extend $u$ and $\bF$ as zero outside $\Omega$ and recall that
the assumptions on the regularity of $\partial \Omega $ allow  us to
do it.  Let $y \in \Omega$ and
$r>0$. Then we calculate:
\begin{align}\label{eqI1}
\nonumber
I:= &\  \frac1{\varphi(y,r)}\int_{\Omega_r(y) } |Du(x)|^{pq}\, dx\\
\nonumber
= &\   \frac1{\varphi(y,r)}\int_{\R^n } |Du(x)|^{pq} \chi_{\Omega_r(y)}(x)\, dx \\
\nonumber
\leq  &\   \frac1{\varphi(y,r)}\int_{\R^n } |Du(x)|^{pq} \chi_{\B_r(y)}(x)\, dx \\
\nonumber
\leq &\   \frac1{\varphi(y,r)}\int_{\R^n } |Du(x)|^{pq}(\M{\chi_{\B_r(y)}}(x))^\sigma\ dx \\
\leq &\   \frac{c}{\varphi(y,r)}\int_{\R^n } |\bF(x)|^{pq}  (\M\chi_{\B_r(y)}(x))^\sigma\, dx,
\end{align}
where we have used the fact that
\begin{equation}\label{eqM1}
 \chi_{\B_r(y)}(x)\leq    \M \chi_{\B_r(y)}(x) =
 \sup_{x\in \B_\rho}\frac{|\B_r(y)\cap \B_\rho|}{|\B_\rho|}\leq    \left(  \M \chi_{ \B_r(y) }
 \right)^\sigma(x) \leq 1
\end{equation}
and the estimate \eqref{2.17BR}.

Simplifying the notations by  writing  $\B=\B_r(y)$ and $2\B=\B_{2r}(y),$ we write
 the dyadic decomposition of $\R^n$ related to $\B$  as
$$
\R^n=2\B\cup \left( \bigcup_{k=1}^\infty 2^{k+1}\B\setminus 2^k\B
\right)\,.
$$
Then \eqref{eqI1} becomes
\begin{align}\label{eqI2}
\nonumber
I  \leq &\  \frac{c}{\varphi(y,r)}\int_{2\B} |\bF(x)|^{pq} (\M\chi_{\B}(x))^\sigma\, dx\\
\nonumber
&\  +  c \sum_{k=1}^\infty   \frac{1}{\varphi(y,r)}
\int_{2^{k+1}\B\setminus 2^k\B }   \frac{1}{\varphi(y,r)}
|\bF(x)|^{pq} (\M\chi_{\B}(x))^\sigma\, dx\\
=: &\  c I_0+c\sum_{k=1}^\infty I_k\,.
\end{align}
Let us estimate the maximal function in the above integrals:
\begin{itemize}
\item[$\circ$] If $x\in 2\B,$ then the supremum in \eqref{eqM1}
is attained when $\B\subseteq
\B_\rho$, and whence $\M\chi_\B(x)\leq 1.$
\item[$\circ$]
If $x\in 2^{k+1}\B\setminus 2^k\B,$ then
$$
 \frac{r^n}{(2^{k+1}r+r)^n}\leq \M\chi_{\B}(x)\leq \frac{r^n}{(2^kr+r)^n},
$$
which permits to compare $ \M\chi_{\B}(x) \sim 2^{-kn}$.
\end{itemize}
Then  making use of \eqref{cond1}, we can estimate \eqref{eqI2} in the following way:
\begin{align}
\label{I0}\nonumber
I_0&\  \leq\frac{c}{\varphi(y,r)}\int_{2\B}|\bF(x)|^{pq}\,dx\\
&\  \leq \frac{c\varkappa_2}{\varphi(y,2r)}\int_{2\B}|\bF(x)|^{pq}\,dx\leq
c \||\bF|^p\|^q_{L^{q,\varphi}(\Omega)}\,,\\
\label{Ik}
\nonumber
I_k &\  \leq    \frac{c}{\varphi(r,y)}\int_{2^{k+1}\B\setminus 2^k\B} |\bF(x)|^{pq} \frac{1}{2^{kn\sigma}}\,dx\\
\nonumber
 &\  \leq \frac{c\varkappa_2^{k+1}}{2^{kn\sigma}} \frac1{\varphi(y,2^{k+1}r)} \int_{2^{k+1}\B} |\bF(x)|^{pq} \,dx\\
 &\  \leq  c\varkappa_2 \frac{\varkappa_2^k}{2^{n\sigma k}} \||\bF|^p\|^q_{L^{q, \varphi}(\Omega)}.
\end{align}
Choosing  $\sigma\in(\frac{\log_2 \varkappa_2}{n}<\sigma<1) $ and 
summing up the integrals $I_k$  we
discover
\begin{equation}\label{sumIk}
I\leq c  \sum_{k=0}^{\infty}\left(  \frac{\varkappa_2}{2^{\sigma n}}\right)^k \
\||\bF|^p\|^q_{L^{q,\varphi}(\Omega)}\leq c
\||\bF|^p\|^q_{L^{q,\varphi}(\Omega)},
\end{equation}
where the constant depends on $n,p,q,$ and $|\Omega|$.

Finally  combining  \eqref{eqI1} and \eqref{sumIk} and taking the
supremum with respect to all balls $\B$, we obtain the desired
estimate \eqref{mainestimate}.
\end{proof}

\section{Generalized H\"older regularity}

If we restrict the  class of weight
functions, then we can obtain,
through  the results of \cite{CP},  a {\it generalized H\"older
regularity} of the solution $u$ to
the problem \eqref{DP}. Suppose that $\varphi(r):\R_+\to \R_+$ is
continuous. The generalized H\"older space $C^{0,\varphi}(\Omega)$ consists  of all continuous functions 
$\varphi(r):\R_+\to \R_+$ for which the following seminorm 
$$
[f]_{C^{0,\varphi}(\Omega)}=\sup_{x,y\in\Omega\atop{x\not=y}}\frac{|f(x)-f(y)|}{\varphi(|x-y|)}
$$
is finite. 
Obviously, if $\lim_{r\to 0_+}\varphi(r)=0$ then $ C^{0,\varphi}(\Omega)$ concontains  uniformly continuous
functions.    Using the technique
of the  {\it rearrangement invariant (r.i.) spaces,} Cianchi and
Pick obtained  a correct  condition
on the weight function  for the embedding of the Sobolev spaces into
generalized H\"older spaces. Precisely, the \cite[Theorem 1.3]{CP}
asserts that if $X(\Omega)$ is r.i. space and $\varphi$ is
a strictly positive continuous
function on $(0,\infty)$, then the following assertions are
equivalent:
\begin{itemize}
\item[(i)] A positive constant $c$ exists such that
$$
\|u\|_{C^{0,\varphi}(Q)}\leq c \| Du \|_{X(Q)}
$$
for every $u\in W^1 X(Q);$
\item[(ii)]\quad $\ds \sup_{0<\rho<|Q|}\,
\frac{1}{\varphi\left(\rho^{\frac{1}{n}}\right)} \left\|
t^{-\frac{1}{n'}} \chi_{(0,\rho)}(t)
\right\|_{\bar{X'}(0,|Q|)}<\infty\,.$
\end{itemize}
Recall that if $X(Q)$ is r.i. space, then $\bar{X}(0,|Q|)$ is its
{\it representation space}, while $X'(Q)$ and $\bar{X'}(0,|Q|)$  are
the topological dual spaces, respectively of $X(Q)$ and
$\bar{X}(0,|Q|)$  and $Q$ is a cube in $\R^n.$   Without
loss  of generality, we can take
$Q$ such that $\Omega\subset Q$ and $|Q|=d_\Omega^n.$  In order to
apply \cite[Theorem 1.3]{CP}, we extend $u$ and $\bF$ as zero
outside $\Omega,$ take $X(Q)=L^{pq}(Q) $ and hence
$\bar X (0,|Q|)= L^{pq}(0,d_\Omega^n),$  $\bar X'=(0,|Q|)= L^{\frac{pq}{pq-1}}(0,d_\Omega^n).$  Then the
condition (ii) becomes
$$
\sup_{0<\rho<d_\Omega^n}\frac1{\varphi(\rho^{1/n})} \left[
\int_0^\rho \big(t^{-\frac1{n'}}\big)^{\frac{pq}{pq-1}}\, dt
\right]^{\frac{pq-1}{pq}}<\infty,
$$
where $\frac{1}{n'}=1-\frac1{n}.$ Direct calculations and change of
the variables   $r= \rho^{\frac{1}{n}}$  give  that for
any fixed $n$ and $p$ the condition (ii) becomes 
\begin{equation}\label{cond4}
\sup_{0<r<d_\Omega}\frac{r^{1-\frac{n}{pq}}}{\varphi(r)}<\infty,\qquad
\forall \ q>\max\left\{ 1, \frac{n}{p}\right\}.
\end{equation}
Then if $\varphi(r) $ satisfies \eqref{cond4}, then  it holds
\begin{equation}\label{eq_Holder}
\| u\|_{C^{0,\varphi}(\Omega)}\leq c\|Du\|_{L^{pq}(\Omega)}\leq  c
 \|\bF\|_{L^{pq}(\Omega)}\,.
\end{equation}

\section{Asymptotically regular problems}\label{secARP}

Consider now the following nonlinear elliptic problem
\begin{equation}\label{DPA}
\begin{cases}
\divv \bb(x, Du)=\divv(|\bF|^{p-2}\bF)  & \text{ in } \Omega\\
u=0 & \text{ on } \partial\Omega,
\end{cases}
\end{equation}
where $\bF\in L^{p,\varphi}(\Omega;\R^n)$ is a given vector-valued
function with $p\in (1,\infty)$ and
$\varphi(x,r)$ satisfying  \eqref{cond1}-\eqref{cond5}.

\begin{defin}
Let  $\ba(x,\xi)$ be a regular operator in the sense of the
Definition~\ref{defRegular}. The operator
$\bb(x,\xi):\R^n\times\R^n\to \R^n,$    satisfying  the
Carath\'eodory conditions, is  asymptotically $\delta$-regular with
$\ba(x,\xi)$  if
\begin{equation}\label{regular1}
\limsup_{|\xi|\to \infty}
\frac{|\bb(x,\xi)-\ba(x,\xi)|}{|\xi|^{p-1}}   \leq \delta,
\end{equation}
uniformly with respect to $x\in \R^n.$
\end{defin}
Let us note that the condition \eqref{regular1} implies that
\begin{equation}
\label{regular2}
|\bb(x,\xi)-\ba(x,\xi)|\leq  \omega(|\xi|)(1+ |\xi|^{p-1}) \textrm{
and } \limsup_{r\to\infty}\omega(r)\leq \delta,
\end{equation}
where  $\omega:\R^+\to \R^+$ is a uniformly bounded function,
defined as
\begin{equation}
\label{xi-osc} \omega(|\xi|)=\sup_{x\in  \R^n}\frac{|\bb(x,\xi)-\ba(x,\xi)|}{|\xi|^{p-1}}.
\end{equation}

The notion of  asymptotic regularity has been introduced in
\cite{CE} assuming that $\lim_{r\to\infty}\omega(r)=0$. Later, this
condition have been relaxed in \cite{BOW} taking boundedness of
$\omega(r)$ on infinity,  including such a way oscillating operators
with small oscillation with respect to $\xi,$ as for example,
$$
\bb(x,\xi)= \ba(x,\xi) +\delta
\sin(|\xi|^2)|\xi|^{p-2\xi}
$$
where $\ba(x,\xi)$ is  regular. In order to obtain
Calder\'on-Zygmund   type  estimates
for  the problem \eqref{DPA}, we
need to transform it in a suitable
regular problem for which we can apply the results obtained in
Section~\ref{sec2}. Let us introduce the  operator
$$
\bc(x,\xi)= \frac{\bb(x,\xi)-\ba(x,\xi)}{|\xi|^{p-1}}, \qquad
|\xi|\not= 0.
$$
Obviously, it is a  Carath\'eodory map. Moreover,  $\limsup_{|\xi|\to
\infty}\bc(x,\xi)\leq \delta$, hence there exists
$M(\delta,\omega)>1$ such that $|\bc(x,\xi)|\leq 2\delta$ if
$|\xi|\geq M,$ uniformly with respect to $x\in\Omega$. Since
$\xi \rightarrow \bc(x,\xi)$ is
continuous for each fixed $x$  over the sphere $|\xi|=M$, we can
apply the classical theory \cite{GT} to construct a harmonic
function in the ball $\B_M(0)=\B_M$. Taking the Poisson kernel
$$
\cP(\xi,\eta)=\frac{M^2-|\xi|^2}{M\omega_{n-1}|\xi-\eta|^n} \quad
\left( \xi\in \B_M, \
\eta\in\partial\B_M\right),
$$
we consider the Poisson integral
$$
P[\bc(x, \cdot)](\xi)=\int_{\partial \B_M}\bc(x,\eta)\cP(\xi,\eta)\,
d\sigma_{\eta} \quad  \left(\xi \in \B_M\right).
$$
Then $\xi \rightarrow P[\bc(x,\cdot)](\xi)$  is harmonic function in $\B_M$ with respect to $\xi$ 
 and coincides with $\bc(x,\xi)$ on $\partial \B_M.$ Define
the function
$$
\tl \bc(x,\xi):=\begin{cases}
\bc(x,\xi) & \text{ if }  |\xi|\geq M\\
P[\bc (x,\cdot)](\xi)   & \text{ if
}  |\xi| < M
\end{cases}
$$
which is a  Carath\'eodory map  in
$\R^n\times \R^n$ such that $|\tl \bc(x,\xi)|\leq 2\delta$ for each
$\xi\in \R^n$ by the maximum principle.

\begin{lem}\label{lem3}
Let $u\in W_0^{1,p}(\Omega)$ be a weak solution of the asymptotical
problem \eqref{DPA}. Then by \cite{BOW}, it is a weak solution of
the problem
\begin{equation}\label{DPA1}
\begin{cases}
\divv \tl\bc(x, Du)=\divv(|\bG|^{p-2}\bG)  & \text{ in } \Omega\\
u=0 & \text{ on } \partial\Omega,
\end{cases}
\end{equation}
where $\bG$ is defined by
$$
\bG=\frac{|\bF|^{p-2}\bF +|Du|^{p-1}\chi_{\{|Du|<M\}} (\tl\bc(x,Du))
-\bc(x,Du) } { | |\bF|^{p-2}\bF +|Du|^{p-1}\chi_{\{|Du|<M\}}
(\tl\bc(x,Du)) -\bc(x,Du)|^{\frac{p-2}{p-1}}  }
$$
if the denominator is different from zero, and $\bG=0$ when the
denominator vanishes. Hence $\bG\in L^p(\Omega;\R^n)$ and
$$
\|\bG\|_{L^p(\Omega)}\leq C( \|\bF\|_{L^p(\Omega)} +1 ).
$$
Since  $|\bF|^p\in L^{q,\varphi}(\Omega)$
for some $q\in[1,\infty)$ and any
fixed $p\in(1,\infty)$, $\bF\in L^{pq,\varphi}(\Omega; \R^n)$. We
are going to show that $\bG\in L^{pq,\varphi}(\Omega; \R^n).$
\end{lem}
\begin{proof}
For $\delta$ small, it  holds $ | \tl\bc(x,Du)  |\leq 2\delta<2$
uniformly with respect to $x$. Then, as in \cite{BOW},
\begin{align*}
&\left||Du|^{p-1}\bc(x,Du)  \right|= |\bb(x,Du)-\ba(x,Du)|\\
&\quad \leq \omega(|Du|)(1+|Du|^{p-1})\leq \|\omega\|_\infty(1+ |Du|^{p-1}  )
\end{align*}
and hence
\begin{equation}\label{eq2}
\left||Du|^{p-1}\chi_{\{|Du|<M  \}}\bc(x,Du)  \right|\leq   2\|
\omega\|_\infty M^{p-1}.
\end{equation}
Accordingly,
 $ |\bG|^p\leq c \left( |\bF|^p+ (1+\| \omega
\|_\infty)^{\frac{p}{p-1}} M^p \right)$. Then
$$
|\bG|^{pq}\leq c\left(
|\bF|^{pq}+(1+\|\omega\|_\infty)^{\frac{pq}{p-1}} M^{pq} \right)
$$
with a constant $c=c(p,q)$. Applying \eqref{cond6}, we get
\begin{align*}
\frac1{\varphi(y,r)}\int_{\Omega_r} |\bG(x)|^{pq}\,dx& \leq
\frac{c}{\varphi(y,r)}\int_{\Omega_r} \bF(x)|^{pq}\, dx+ c\frac{|\Omega_r|}{\varphi(y,r)}\\
&\leq c\left(\|\bF\|^{pq}_{L^{pq,\varphi}} +1  \right)
\end{align*}
with a constant depending on known quantities and
$\|\omega\|_\infty.$ Taking the supremum over $y\in \Omega$ and
$r\in(0,\diam\Omega]$, we get
\begin{equation}\label{eqG}
\|\bG\|^{pq}_{L^{pq,\varphi}(\Omega)}\leq
(\|\bF\|^{pq}_{L^{pq,\varphi}}+1).
\end{equation}
\end{proof}
The desired regularity requirements
of the operator $\tl\bc$ follows by the next lemma that is proved in
\cite{BOW}.
\begin{lem}\label{lem2}
Assume that $\bb(x,\xi)$ is asymptotically $\delta$-regular with $\ba(x,\xi).$ Then
\begin{enumerate}
\item $ {\tl\bc}(x,\xi)$ is regular if $0<\delta<\min \left\{ \frac{\Lambda}{4n(|p-2|+1)},1
\right \}$.
\item    $ {\tl\bc}(x,\xi)$ is $(5\delta,R)$-vanishing.
\end{enumerate}
\end{lem}

We are now ready to state and prove
our desired Calder\'{o}n-Zygmund theory with the desired estimate
for the asymptotically regular problem \eqref{DPA} in generalized
Morrey spaces.

\begin{thm}
\label{thm2} For any $q\in[1,\infty)$ and $\varphi$ satisfying
\eqref{cond1}-\eqref{cond5}, assume  $|\bF|^p\in
L^{q,\varphi}(\Omega;\R^n)$.  Then
there exists $\delta=\delta(n,p,q, \gamma, L,\varphi)>0$ such that if $\bb(x,\xi)$ is asymptotically
$\delta$-regular with $\ba(x,\xi)$
satisfying and \eqref{vanishing} and
if $\Omega$ satisfies \eqref{Reifenberg}, then any weak solution
$u\in W_0^{1,p}(\Omega)$ of \eqref{DPA}   verifies
\begin{equation}\label{CZAs}  
\| |Du|^p \|_{L^{q,\varphi}(\Omega)}\leq c\left(
\||\bF|^p\|_{L^{q,\varphi}(\Omega)}  +1\right)
\end{equation}
for some positive constant
$c=c(n,p,q, \gamma, L,\varphi,|\Omega|)$.
\begin{proof}
Let $\delta_0$ be the constant from Theorem~\ref{MainTh}  and let
$$
\delta_1=\min\left\{ \frac{\lambda}{4n(|p-2|+1)},1 \right\}.
$$
We set $\delta=\frac15\min\{\delta_0,\delta_1  \}>0$. Let $u\in
W_0^{1,p}(\Omega)$ be a weak solution of \eqref{DPA}.  Since $\ba $
is $(\delta_0,R)$-vanishing, it
follows from   Lemma~\ref{lem2} that $\bc$ is $(5\delta,R)$ vanishing
with $5\delta <\delta_0$ and hence also $(\delta_0,R)$-vanishing.
According to  Lemma~\ref{lem3}, we have that $|\bG|^p\in
L^{q,\varphi}(\Omega)$ hence by the Theorem~\ref{MainTh}, $|Du|^p\in
L^{q,\varphi}(\Omega)$ and the following estimate holds
$$
\||Du|^p\|_{L^{q,\varphi}(\Omega)}\leq c\||\bG|^p\|_{L^{q,\varphi}(\Omega)}\leq
c \left( \||\bF|^p\|_{L^{q,\varphi}(\Omega)}+1 \right)\,.
$$

\end{proof}

\end{thm}

\end{document}